\documentclass[11pt]{article}
       \usepackage{amsfonts}
       \usepackage{stmaryrd}
       \usepackage{arydshln}
       \usepackage{graphicx} 
       \usepackage{tikz}
       \usepackage[all]{xy}
       \usepackage{multirow}
       \usepackage{latexsym,amssymb,mathrsfs,fancyhdr}
       \font\tenmsb=msbm10
       \font\sevenmsb=msbm7
       \font\fivemsb=msbm5
       \catcode`\@=11
       \ifx\amstexloaded@\relax\catcode`\@=\active
       \endinput\else\let\amstexloaded@\relax\fi
       \def\spaces@{\space\space\space\space\space}
       \def\spaces@@{\spaces@\spaces@\spaces@\spaces@\spaces@}
       \def\space@.  {\futurelet\space@\relax}
       \space@.   %
       \def\Err@#1{\errhelp\defaulthelp@\errmessage{AmS-TeX error: #1}}
       \def\relaxnext@{\let\next\relax}
       \def\accentfam@{7}
       \def\noaccents@{\def\accentfam@{0}}
       \def\Cal{\relaxnext@\ifmmode\let\next\Cal@\else
       \def\next{\Err@{Use \string\Cal\space only in math mode}}\fi\next}
       \def\Cal@#1{{\Cal@@{#1}}}
       \def\Cal@@#1{\noaccents@\fam\tw@#1}
       \def\Bbb{\relaxnext@\ifmmode\let\next\Bbb@\else
       \def\next{\Err@{Use \string\Bbb\space only in math mode}}\fi\next}
       \def\Bbb@#1{{\Bbb@@{#1}}}
       \def\Bbb@@#1{\noaccents@\fam\msbfam#1}
       \newfam\msbfam
       \textfont\msbfam=\tenmsb
       \scriptfont\msbfam=\sevenmsb
       \scriptscriptfont\msbfam=\fivemsb
\makeatletter 
\@addtoreset{equation}{section}
\makeatother  

\usepackage{dsfont}
\usepackage[german,english]{babel}
\usepackage{amsmath,amssymb}
\usepackage[square, comma, sort&compress, numbers]{natbib}
\usepackage{cases}
\usepackage{mathrsfs}
\usepackage{amsmath}
\usepackage{amsthm}
\usepackage{amsfonts}
\usepackage{amssymb}
\usepackage{latexsym}
\usepackage{fancyhdr}
\usepackage{extarrows}
\usepackage{geometry}
\newtheorem{thm}{Theorem}[section]

\newtheorem{lem}[thm]{Lemma}
\newtheorem{rem}[thm]{Remark}
\newtheorem{iteration lemma}[thm]{iteration Lemma}
\newtheorem{cor}[thm]{Corollary}

\newtheorem{eg}[thm]{Example}

\newtheorem*{acknowledgements*}{ACKNOWLEDGEMENTS}

\begin{document}

\setlength{\columnsep}{5pt}
\title{\bf 
Continuity of the  core-EP inverse and its applications
}

\author{Yuefeng Gao$^{1,2}$\footnote{  E-mail: yfgao91@163.com},
 \ Jianlong Chen$^{1}$\footnote{E-mail: jlchen@seu.edu.cn}, 
  \ Pedro Patr\'{i}cio$^{2}$ 
  \footnote{ E-mail: pedro@math.uminho.pt}
 \\\\
$^{1}$School of Mathematics, Southeast University, Nanjing 210096, China~~~~~~\!~~~~~~~~~~~~\\
$^{2}$CMAT-Centro de Matem\'{a}tica, Universidade do Minho, Braga 4710-057, Portugal}

  \date{}

\maketitle
\begin{quote}
{\textbf{Abstract: }\small  
In this paper, firstly we study the continuity of the core-EP inverse without explicit error bounds by virtue of two methods.
One is the rank equality, followed from the classical generalized inverse. The other one is matrix decomposition. 
The continuity of the core inverse can be derived as a particular case.
Secondly, we study perturbation bounds for the core-EP inverse  under prescribed conditions. Perturbation bounds for the core inverse can be derived as  a particular case. Also, as corollaries, the sufficient (and necessary) conditions for  the continuity of the core-EP inverse  are obtained. Thirdly, a numerical example is illustrated to compare the derived upper bounds.  Finally, an application to semistable matrices is provided.

\textbf {Keywords:} {\small   core-EP inverse, pseudo core inverse, continuity, perturbation bound}

\textbf {AMS Subject Classifications:} 15A09;  34D10; 65F35}
\end{quote}

\section{ Introduction}
It is known that the inverse of a non-negative matrix is a continuous function.
However, in general, the operations of generalized inverses such as Moore-Penrose inverse, Drazin inverse, weighted Drazin inverse, generalized inverse $A^{(2)}_{T,S}$, core inverse are not continuous \cite{C1975,C1977,C1980,S1969,W2003}. 
It is of interest to know whether the continuity of the core-EP inverse holds.
In this note, we will answer this question.

Throughout this paper, $\mathbb{C}^{n}$ denotes the sets of all $n$-dimensional column vectors
  and $\mathbb{C}^{m\times n}$ is used to denote the set of all $m\times n$ complex matrices. 
For each complex matrix $A\in \mathbb{C}^{m\times n}$, $A^*$ 
denotes the conjugate transpose
of $A$, $\mathcal{R}(A)$ and $\mathcal{N}(A)$ denote the range (column space) and null space of $A$ respectively.
The Moore-Penrose inverse of $A$, denoted by $A^{\dag}$, is the unique solution to 
$$AXA=A,~XAX=X,~(AX)^*=AX~\text{and}~(XA)^*=XA.$$  
Let $A\in \mathbb{C}^{n\times n}$, the index of $A$, denoted by ind$(A)$, is the smallest non-negative integer $k$ for which rank$(A^k)=$rank$(A^{k+1})$. The 
 Drazin inverse of $A$ is the unique solution to system
 $$AXA^k=A^k,~XAX=X~\text{and}~AX=XA.$$

Recall that the core-EP inverse was proposed by Manjunatha Prasad and Mohana \cite{2014D} for a square  matrix of arbitrary index,
as an extension of  the core inverse  restricted to a square  matrix of index one
in \cite{B2010}. Then, Gao and Chen \cite{G2018} characterized the core-EP inverse (also known as the pseudo core inverse) in terms of three equations. Let $A\in \mathbb{C}^{n\times n}$ with ind$(A)=k$, the core-EP inverse of $A$, denoted by $A^{\tiny{\textcircled{\tiny \dag}}}$, is the unique solution to the system
\begin{equation}
XA^{k+1}=A^{k},~AX^2=X,~(AX)^*=AX.
\end{equation}
The core-EP inverse is an
outer inverse (resp. \{2\}-inverse), i.e., $A^{\tiny{\textcircled{\tiny \dag}}}AA^{\tiny{\textcircled{\tiny \dag}}}=A^{\tiny{\textcircled{\tiny \dag}}}$, see \cite{G2018}.  
If $k=1$, then the core-EP inverse of $A$ is the core inverse of $A$. denoted by $A^{\tiny{\textcircled{\tiny \#}}}$ (see \cite{B2010}). 
\begin{lem}\label{1.1}\emph{\cite{G2018}}~Let $A\in\mathbb{C}^{n\times n}$ with $\mathrm{ind}(A)=k$. Then we have the following facts:

$(1)~A^{{\tiny{\textcircled{\tiny \dag}}}}=A^DA^k(A^k)^{\dag},$

$(2)~A^k(A^k)^{\dag}=A^j(A^j)^{\dag}~(j\geq k),$

$(3)~A^D=(A^{{\tiny{\textcircled{\tiny \dag}}}})^{k+1}A^k$.
\end{lem}
\noindent From Lemma \ref{1.1}, it follows that
 \begin{equation}
 A^{{\tiny{\textcircled{\tiny \dag}}}}=A^DA^k(A^k)^{\dag}=A^DA^{k+1}(A^{k+1})^{\dag}=A^k(A^{k+1})^{\dag}.
 \end{equation}


Recall that the Euclidean vector norm is defined by 
$$\|x\|^2=x^*x~\text{for~any}~x\in \mathbb{C}^n,$$ 
the  spectral norm of a matrix $A\in \mathbb{C}^{n\times n}$ is defined by 
$$\|A\|=\sup\limits_{x\neq 0}\frac{\|Ax\|}{\|x\|},$$
and
\begin{equation*}
\begin{aligned}
&\|Ax\|\leq \|A\|\|x\|~\text{for~all}~A\in \mathbb{C}^{n\times n}~\text{and~all}~x\in \mathbb{C}^n,\\
&\|AB\|\leq \|A\|\|B\|~\text{for~all}~A,~B\in \mathbb{C}^{n\times n},\\
&\|A^*\|=\|A\|~\text{for~all}~A\in \mathbb{C}^{n\times n}.
\end{aligned}
\end{equation*}
For a non-singular matrix $A$, $\kappa(A)=\|A\|\|A^{-1}\|$ denotes the condition number of $A$. As usual, this is generalized to the core-EP condition number $\kappa_{{\tiny{\textcircled{\tiny \dag}}}}(A)=\|A\|\|A^{{\tiny{\textcircled{\tiny \dag}}}}\|$  if $A$ is singular.
\begin{lem}\emph{\cite{S1969}}\label{1.2}~Let $A\in\mathbb{C}^{n\times n}$ with $\|A\|<1$. Then $I+A$ is non-singular and
$$\|(I+A)^{-1}\|\leq (1-\|A\|)^{-1}.$$
\end{lem}


The paper is organized as follows. In Section 2, the continuity of the core-EP inverse without explicit error bounds is investigated by means of a rank equation  and a matrix decomposition respectively.  The continuity of the core inverse are obtained as corollaries.
In Section 3, perturbation bounds for the core-EP inverse  are investigated respectively under three  cases: 

$(1)~\mathcal{R}(E)\subseteq\mathcal{R}(A^k)~\text{and}~\mathcal{N}(A^{k*})\subseteq\mathcal{N}(E)$, where $k=\mathrm{ind}(A)$.

$(2)~AA^{{\tiny{\textcircled{\tiny \dag}}}}=(A+E)(A+E)^{{\tiny{\textcircled{\tiny \dag}}}}~\text{and}~A^{{\tiny{\textcircled{\tiny \dag}}}}A=(A+E)^{{\tiny{\textcircled{\tiny \dag}}}}(A+E),$

$(3)~\mathrm{rank}\left(A^k\right)=\mathrm{rank}\left((A+E)^k\right),$ where $k=\mathrm{max}\{\mathrm{ind}(A),\mathrm{ind}(A+E)\}$.\\
Notice that $(1)$ is equivalent to $$(4)~E=A^{{\tiny{\textcircled{\tiny \dag}}}}AE=EAA^{{\tiny{\textcircled{\tiny \dag}}}}.$$
The relation scheme of (1)-(3) states  as follows : in general, (1) may not imply (2), see Example \ref{1.3}; (2) may not imply (1), see Example \ref{1.4}; (1) may not imply (3), and (3) may not imply (1), see Examples \ref{1.3} and \ref{1.5};
(2) implies (3), but (3) may not imply (2), see Example \ref{1.5}. 
\begin{eg}\label{1.3}~Let $A=\begin{bmatrix}
                                  1&0\\
                                  0&0.1
                                 \end{bmatrix},~E=\begin{bmatrix}
                                  0&0\\
                                  0&-0.1
                                 \end{bmatrix}$. Then $E=A^{{\tiny{\textcircled{\tiny \dag}}}}AE=EAA^{{\tiny{\textcircled{\tiny \dag}}}}$. However,  $AA^{{\tiny{\textcircled{\tiny \dag}}}}\neq(A+E)(A+E)^{{\tiny{\textcircled{\tiny \dag}}}},~A^{{\tiny{\textcircled{\tiny \dag}}}}A\neq(A+E)^{{\tiny{\textcircled{\tiny \dag}}}}(A+E)$ and $\mathrm{rank}(A)\neq \mathrm{rank}(A+E)$.
\end{eg}

\begin{eg}\label{1.4}
 let $A=\begin{bmatrix}
                                  1&0&0\\
                                  0&0&0\\
                                  0&0&0
                                 \end{bmatrix},~E=\begin{bmatrix}
                                  0.1&0&0\\
                                  0&0&0.1\\
                                  0&0&0
                                 \end{bmatrix}$. Then
                                 $AA^{{\tiny{\textcircled{\tiny \dag}}}}=(A+E)(A+E)^{{\tiny{\textcircled{\tiny \dag}}}},~A^{{\tiny{\textcircled{\tiny \dag}}}}A=(A+E)^{{\tiny{\textcircled{\tiny \dag}}}}(A+E)$. However,  $AA^{{\tiny{\textcircled{\tiny \dag}}}}E\neq E$.
\end{eg}

\begin{eg}\label{1.5}~let $A=\begin{bmatrix}
                                  1&0&0\\
                                  0&0&0\\
                                  0&0&0
                                 \end{bmatrix},~E=\begin{bmatrix}
                                 0.1&0.1&0\\
                                  0&0&0\\
                                  0&0&0
                                 \end{bmatrix}$. Then $\mathrm{rank}(A)=\mathrm{rank}(A+E)$. However, $A^{{\tiny{\textcircled{\tiny \dag}}}}A\neq(A+E)^{{\tiny{\textcircled{\tiny \dag}}}}(A+E)$ and $E\neq EAA^{{\tiny{\textcircled{\tiny \dag}}}}$.
\end{eg}

\noindent Among the above conditions, (3) would be the weakest condition to consider the  perturbation bounds for the core-EP inverse. Although (2) is stronger than (3), yet (2) in conjunction  with other restrictions on $A,~E$ would help to acquire a better error bound.  Thus (1)-(3) are all worth to be studied. 
As  special cases, perturbation bounds for the core inverse  are obtained. 
Meanwhile, the sufficient (and necessary) conditions for which the operation  of the core-EP inverse is continuous  are derived as natural  outcomes.  In Section 4, a numerical example is illustrated to compare the upper bounds for   $\frac{\|(A+E)^{{\tiny{\textcircled{\tiny \dag}}}}-A^{{\tiny{\textcircled{\tiny \dag}}}}\|}{\|A^{{\tiny{\textcircled{\tiny \dag}}}}\|}$ by using derived results in Section 3. It turns out that the bounds in case (1) and (2) are slightly better than that in case (3). 
 In Section 5, an  application to semistable matrices is provided.

\section{Continuity of the core-EP inverse}
The following example shows that the core-EP inverse of a square matrix is not continuous.
\begin{eg}~Let $A=\left[\begin{matrix}
                                 1/j &1& 0 &0\\
                                 0&0&0&0\\
                                 0&0&0&1\\
                                 0&0&0&0
                               \end{matrix}\right]$ and 
                               $A=\left[\begin{matrix}
                                 0 &1 &0 &0\\
                                 0&0&0&0\\
                                 0&0&0&1\\
                                 0&0&0&0
                               \end{matrix}\right]$. Then $A_j\rightarrow A$. However, 
 $$A_j^{{\tiny{\textcircled{\tiny \dag}}}}
 =\left[\begin{matrix}
                   j &0 &0 &0\\
                0&0&0&0\\
              0&0&0&0\\
              0&0&0&0
     \end{matrix}\right]\nrightarrow 0=A^{{\tiny{\textcircled{\tiny \dag}}}}.$$

\end{eg}
In the rest of  this section, we consider the necessary and sufficient   conditions for which  the core-EP inverse has the continuity property. 
\subsection{Rank equality method}
In \cite{C1975}, the continuity of classical generalized inverses are studied by means of  rank equalities. 
Analogously, we consider the continuity of the core-EP inverse.  
\begin{lem}\label{2.2}\emph{\cite{C1975}}~Let $\{A_j\}\subseteq \mathbb{C}^{m\times n},$ $A\in \mathbb{C}^{m\times n}$ with $A_{j}\rightarrow A$. Then $A_{j}^{\dag}\rightarrow A^{\dag}$ if and only if there exists $j_0$ such that $\mathrm{rank}(A_j) =\mathrm{rank}(A)$ for $j\geq j_0$.
\end{lem}

\begin{lem}\label{2.3}\emph{\cite{C1975}}~Let $\{A_j\}\subseteq \mathbb{C}^{n\times n},$ $A\in \mathbb{C}^{n\times n}$ with $A_{j}\rightarrow A$. Then $A_{j}^{D}\rightarrow A^{D}$ if and only if there exists $j_0$ such that $\mathrm{rank}(A_j^{\mathrm{ind}(A_j)}) = \mathrm{rank}(A^{\mathrm{ind}(A)})$ for $j\geq j_0$.
\end{lem}

\begin{lem}\label{2.4}\emph{\cite{C1975}}~Let $\{A_j\}\subseteq \mathbb{C}^{n\times n},$ $A\in \mathbb{C}^{n\times n}$ with $A_{j}\rightarrow A$, $A_{j}^{D}\rightarrow A^{D}$. Then  there exists $j_0$ such that $\mathrm{ind}(A) \leq \mathrm{ind}(A_j)$ for $j\geq j_0$.
\end{lem}

Analogous to Lemma \ref{2.4}, we  establish  a similar result for the core-EP inverse.

\begin{lem}\label{2.5}~Let $\{A_j\}\subseteq \mathbb{C}^{n\times n},$ $A\in \mathbb{C}^{n\times n}$ with $A_{j}\rightarrow A$, $A_{j}^{{\tiny{\textcircled{\tiny \dag}}}}\rightarrow A^{{\tiny{\textcircled{\tiny \dag}}}}$. Then  there exists $j_0$ such that $\mathrm{ind}(A) \leq \mathrm{ind}(A_j)$ for $j\geq j_0$.
\end{lem}

\begin{proof}~The proof is similar to the Drazin inverse case.  For completeness, let us give the proof. 

Suppose that $A_{j}\rightarrow A$ and $A_{j}^{{\tiny{\textcircled{\tiny \dag}}}}\rightarrow A^{{\tiny{\textcircled{\tiny \dag}}}}$. Let $\{A_{j_i}\}$ be a subsequence with constant index $k$ of $\{A_{j}\}$. Then $A_{j_i}^{{\tiny{\textcircled{\tiny \dag}}}}(A_{j_i})^{k+1}=(A_{j_i})^k$. By taking limits, we derive that 
$$A^{{\tiny{\textcircled{\tiny \dag}}}}A^{k+1}=A^k.$$
Hence ind$(A)\leq k$. Since the index function takes only finitely many values between 0 and $n$, we obtain that there exists a $j_{0}$ such that $$\mathrm{ind}(A) \leq \mathrm{ind}(A_j)~\text{for}~j\geq j_0.$$
\end{proof}
Making an integral application of Lemmas \ref{2.2}-\ref{2.5}, we derive the following result.
\begin{thm}\label{2.6}~Let $\{A_j\}\subseteq \mathbb{C}^{n\times n},$ $A\in \mathbb{C}^{n\times n}$ with $A_{j}\rightarrow A$. Then the following are equivalent:

$(1)~A_{j}^{{\tiny{\textcircled{\tiny \dag}}}}\rightarrow A^{{\tiny{\textcircled{\tiny \dag}}}};$

$(2)~A_j^{D}\rightarrow A^{D};$ 

$(3)$~there exists $j_0$ such that $\mathrm{rank}(A_j^{\mathrm{ind}(A_j)}) = \mathrm{rank}(A^{\mathrm{ind}(A)})$ for $j\geq j_0;$

$(4)$~there exists $j_0$ such that $\mathrm{rank}(A_j^{\mathrm{ind}(A_j)})= \mathrm{rank}(A^{\mathrm{ind}(A_j)})= \mathrm{rank}(A^{\mathrm{ind}(A)})$ for $j\geq j_0.$
\end{thm}

\begin{proof}~$(1)\Rightarrow(2)$ From Lemma \ref{1.1}, it follows that $A_j^D=(A_{j}^{{\tiny{\textcircled{\tiny \dag}}}})^{\mathrm{ind}(A_j)+1}A_j^{\mathrm{ind}(A_j)}$, which  converges to $(A^{{\tiny{\textcircled{\tiny \dag}}}})^{\mathrm{ind}(A_j)+1}A^{\mathrm{ind}(A_j)}$. By Lemma \ref{2.5}, there exists $j_0$ such that $\mathrm{ind}(A_j)\geq \mathrm{ind}(A)$ for $j\geq j_0$.
Then for $j\geq j_0$, 
 $$(A^{{\tiny{\textcircled{\tiny \dag}}}})^{\mathrm{ind}(A_j)+1}A^{\mathrm{ind}(A_j)}=(A^{{\tiny{\textcircled{\tiny \dag}}}})^
{\mathrm{ind}(A)+1}A^{\mathrm{ind}(A)}=A^D.$$ Namely $A_j^D \rightarrow A^D$.

$(2)\Leftrightarrow(3)$ It is clear by Lemma \ref{2.3}.

$(3)\Rightarrow(4)$ Since $\mathrm{rank}(A_j^{\mathrm{ind}(A_j)}) = \mathrm{rank}(A^{\mathrm{ind}(A)})$ for $j\geq j_0$, then  
$A_j^D \rightarrow A^D~\text{by~Lemma}$~\ref{2.3}. 
Thus, there exists  $j_1$ such that $\mathrm{ind}(A_j)\geq \mathrm{ind}(A)$ for $j\geq j_1$ in view of Lemma \ref{2.4}. Therefore $$A^{\mathrm{ind}(A_j)}=A^{\mathrm{ind}(A)}A^{\mathrm{ind}(A_j)-\mathrm{ind}(A)}~\text{and}~A^{\mathrm{ind}(A)}=(A^D)^{\mathrm{ind}(A_j)-\mathrm{ind}(A)}A^{\mathrm{ind}(A_j)}$$
for $j\geq j_0=$max$\{j_0,j_1\}$, which imply that $\mathrm{rank}(A^{\mathrm{ind}(A_j)})=\mathrm{rank}(A^{\mathrm{ind}(A)})$ for $j\geq j_0$.

$(4)\Rightarrow(1)$ From the assumption, we derive
\begin{equation*}
\begin{aligned}
&A_j^{D}\rightarrow A^D~\text{by~applying~Lemma}~\ref{2.3},\\
&~\text{there~exists~}j_1~\text{such~that}~\mathrm{ind}(A_j)\geq \mathrm{ind}(A)~\text{for}~j\geq j_1~\text{by~applying~Lemma}~\ref{2.4},\\
&(A_j^{\mathrm{ind}(A_j)})^{\dag}\rightarrow (A^{\mathrm{ind}(A_j)})^{\dag}~\text{by~applying~Lemma}~\ref{2.2}.
\end{aligned}
\end{equation*}
In light of Lemma \ref{1.1}, 
$A_j^{{\tiny{\textcircled{\tiny \dag}}}}=A_j^DA_j^{\mathrm{ind}(A_j)}(A_j^{\mathrm{ind}(A_j)})^{\dag}\rightarrow A^DA^{\mathrm{ind}(A_j)}(A^{\mathrm{ind}(A_j)})^{\dag}$.
Since $\mathrm{ind}(A_j)\geq \mathrm{ind}(A)$ for $j\geq j_1$, then $A^{\mathrm{ind}(A_j)}(A^{\mathrm{ind}(A_j)})^{\dag}=A^{\mathrm{ind}(A)}(A^{\mathrm{ind}(A)})^{\dag}$ for $j\geq j_1$. Hence,  
$A_j^{{\tiny{\textcircled{\tiny \dag}}}}\rightarrow A^DA^{\mathrm{ind}(A)}(A^{\mathrm{ind}(A)})^{\dag}=A^{{\tiny{\textcircled{\tiny \dag}}}}.$
\end{proof}

The continuity of   the core inverse can be derived as a particular case $\mathrm{ind}(A)=\mathrm{ind}(A_j)=1$ in Theorem \ref{2.6}.
\begin{cor}~If $\{A_j\}\subseteq \mathbb{C}^{n\times n},$ $A\in \mathbb{C}^{n\times n}$ and $A_{j}\rightarrow A$. Then the following are equivalent:

$(1)~A_{j}^{{\tiny{\textcircled{\tiny \#}}}}\rightarrow A^{{\tiny{\textcircled{\tiny \#}}}};$

$(2)~A_j^{\#}\rightarrow A^{\#};$

$(3)$~there exists $j_0$ such that $\mathrm{rank}(A_j) = \mathrm{rank}(A)$ for $j\geq j_0.$
\end{cor}

\subsection{Matrix decomposition method}

In \cite{C1977}, Pierce decomposition is used to study the continuity of the Moore-Penrose inverse. However this approach is  not suitable for the core-EP inverse since the core-EP inverse is not an inner inverse. As an alternative, we make use of
the core-EP decomposition. 

Recall that the core-EP decomposition  \cite{W2016} of $A$ is 
\begin{equation}
A=U\begin{bmatrix}
           T&S\\
           0&N
          \end{bmatrix}U^*=U\begin{bmatrix}
           T&S\\
           0&0
          \end{bmatrix}U^*+U\begin{bmatrix}
           0&0\\
           0&N
          \end{bmatrix}U^*=A_1+A_2,
 \end{equation} 
 where $U$ is unitary, $T$ is non-singular and $N$ is nilpotent with index $k$, in which case,
          $$A^{{\tiny{\textcircled{\tiny \dag}}}}=A_1^{{\tiny{\textcircled{\tiny \#}}}}=U\begin{bmatrix}
           T^{-1}&0\\
           0&0
          \end{bmatrix}U^*.$$

Fix $A\in \mathbb{C}^{n\times n}$ with $\mathrm{ind}(A)=k$ and consider the following equations
\begin{equation}
XA^{k+1}-A^k=E_1,~AX^2-X=E_2~\text{and}~AX-(AX)^*=E_3.
\end{equation}
Here $X$ may be thought of as an approximation  and the $E_i~(i=1,~2,~3)$ as error terms. Let $X=A^{{\tiny{\textcircled{\tiny \dag}}}}+F.$ Then (2.2) becomes
\begin{equation}
FA^{k+1}=E_1,~~~~~~
\end{equation}
\begin{equation}
~~~~~~~~~~~~~~~~~~~~~AA^{{\tiny{\textcircled{\tiny \dag}}}}F+AFA^{{\tiny{\textcircled{\tiny \dag}}}}+AF^2-F=E_2,
\end{equation}
\begin{equation}
AF-(AF)^*=E_3.
\end{equation}

Suppose that $F=U\begin{bmatrix}
           X_1&X_2\\
           X_3&X_4
          \end{bmatrix}U^*$. According to (2.3),       
\begin{equation*}
U\begin{bmatrix}
           X_1T^k&X_1\sum\limits_{i+j=k}T^{i}SN^j\\
           X_3T^k&X_3\sum\limits_{i+j=k}T^{i}SN^j
          \end{bmatrix}U^*=E_1,  
          \end{equation*}
       \begin{equation}   
      \text{i.e.,}~\begin{bmatrix}
           X_1&\Theta_2\\
           X_3&\Theta_4
          \end{bmatrix}=U^*E_1U\begin{bmatrix}
           (T^k)^{-1}&0\\
           0&I
          \end{bmatrix},
\end{equation}
$\text{where}~ 
\Theta_2=X_1\sum\limits_{i+j=k}T^{i}SN^j~\text{and}~
\Theta_4=X_3\sum\limits_{i+j=k}T^{i}SN^j.$
\\
Then according to (2.4),    
\begin{equation}
U\begin{bmatrix}
           \Delta_1&\Delta_2\\
           \Delta_3&NX_3X_2+NX_4^2-X_4
          \end{bmatrix}U^*=E_2,
\end{equation} 
where 
\begin{equation*}
\begin{aligned}
&\Delta_1=TX_1T^{-1}+SX_3T^{-1}+TX_1^2+SX_3X_1+TX_2X_3+SX_4X_3,~~~~~~~~~~~~~~~~~\\
&\Delta_2=TX_1X_2+SX_3X_2+TX_2X_4+SX_4^2,\\
&\Delta_3=NX_3T^{-1}-X_3+NX_3X_1+NX_4X_3.  
\end{aligned}
\end{equation*}
\\
Finally according to (2.5),   
\begin{equation}
U\begin{bmatrix}
           \Gamma_1&TX_2+SX_4-(NX_3)^*\\
           \Gamma_3&\Gamma_4
          \end{bmatrix}U^*=E_3,
\end{equation}
where $ \Gamma_1=T_1X_1+SX_3-(T_1X_1+SX_3)^*$, $ \Gamma_3=NX_3-(TX_2+SX_4)^*$ and $\Gamma_4=NX_4-(NX_4)^*$.

If $E_i\rightarrow 0$, by applying  (2.6)-(2.8), then
\begin{equation}
X_1\rightarrow 0,~X_3\rightarrow 0,
\end{equation} 
\begin{equation}
NX_4^2-X_4\rightarrow 0, 
\end{equation}  
\begin{equation}
X_2+T^{-1}SX_4\rightarrow 0.  
\end{equation} 
From (2.9), it follows that 
\begin{equation}
X_4\rightarrow NX_4^2\rightarrow N^{k}X_4^{k+1}=0.
\end{equation}
Plug $X_4\rightarrow 0$ into (2.10), giving 
\begin{equation} X_2\rightarrow 0.
\end{equation}
In view of (2.9), (2.12) and (2.13),  $F\rightarrow 0$.
Hence we have the following result.
\begin{thm}\label{2.8}~Let $A\in \mathbb{C}^{n\times n}$ with $\mathrm{ind}(A)=k$. If $\{X_j\}$ is a sequence of $n\times n$ matrices such that the sequences 
$\{X_jA^{k+1}-A^{k}\},~\{AX_j^2-X_j\}~\text{and}~\{AX_j-(AX_j)^*\}$ all converge to zero, then $\{X_j\}$ converges to $A^{{\tiny{\textcircled{\tiny \dag}}}}$.
\end{thm}

A consequence of Theorem \ref{2.8} is that it makes sense to check a computed $\hat{A}^{{\tiny{\textcircled{\tiny \dag}}}}$ exactly by using the system (1.1) if $A$ is known.

The case of the core inverse can be derived by letting $k=1$ in 
Theorem \ref{2.8}.
\begin{cor}~Let $A\in \mathbb{C}^{n\times n}$ with $\mathrm{ind}(A)=1$. If $\{X_j\}$ is a sequence of $n\times n$ matrices such that the sequences 
$\{X_jA^{2}-A\},~\{AX_j^2-X_j\}~\text{and}~\{AX_j-(AX_j)^*\}$ all converge to zero, then $\{X_j\}$ converges to $A^{{\tiny{\textcircled{\tiny \#}}}}$.
\end{cor}


\section{Perturbation bounds}
In this section, we consider perturbation bounds for the core-EP inverse under prescribed conditions. We refer readers to \cite{WL2003,C2008,V2009,W1997,C2000,S1969,R1982,X2010} for a deep study of the perturbation bounds for classical generalized inverses and refer readers to \cite{M2018} for  the core inverse.

\subsection{The case:  $\mathcal{R}(E)\subseteq\mathcal{R}(A^k)~\text{and}~\mathcal{N}(A^{k*})\subseteq\mathcal{N}(E)$}

 In this part, we study perturbation bounds for $(A+E)^{{\tiny{\textcircled{\tiny \dag}}}}$ in the case: $$\mathcal{R}(E)\subseteq\mathcal{R}(A^k),~\mathcal{N}(A^{k*})\subseteq\mathcal{N}(E),~\text{where}~k=\mathrm{ind}(A).$$ 
After which, a sufficient condition for the continuity of  the core-EP inverse  is derived naturally. 
 \begin{thm}\label{theorem:3.4}Let $A,~E\in\mathbb{C}^{n\times n}$ and~$k=\mathrm{ind}(A)$. If $\mathcal{R}(E)\subseteq\mathcal{R}(A^k)$, $\mathcal{N}(A^{k*})\subseteq\mathcal{N}(E)$ and $\|A^{{\tiny{\textcircled{\tiny \dag}}}}E\|<1$. Then
 \begin{equation}
 (A+E)^{{\tiny{\textcircled{\tiny \dag}}}}=(I+A^{{\tiny{\textcircled{\tiny \dag}}}}E)^{-1}A^{{\tiny{\textcircled{\tiny \dag}}}}
 \end{equation}
and
  \begin{equation}
 \frac{\|(A+E)^{{\tiny{\textcircled{\tiny \dag}}}}-A^{{\tiny{\textcircled{\tiny \dag}}}}\|}{\|A^{{\tiny{\textcircled{\tiny \dag}}}}\|}\leq\frac{\|A^{{\tiny{\textcircled{\tiny \dag}}}}E\|}{1-\|A^{{\tiny{\textcircled{\tiny \dag}}}}E\|}.
 \end{equation}
 \end{thm}
 
\begin{proof}In view of  (2.5), there exist unitary matrices $U$ such that $$A=U\left[\begin{matrix}
                  T& S\\
                  0 &N
                  \end{matrix}
                  \right]U^*.$$
 Let $$E=U\left[\begin{matrix}
                  E_{1}& E_{2}\\
                  E_{3} &E_{4}
                  \end{matrix}
                  \right]U^*.$$   
From   the assumption    $\mathcal{R}(E)\subseteq\mathcal{R}(A^k)$, it follows that $E=A^{{\tiny{\textcircled{\tiny \dag}}}}AE$, which implies
$$E_{3}=0~\text{and}~E_{4}=0.$$
Then from  the assumption $\mathcal{N}(A^{k*})\subseteq\mathcal{N}(E)$,  we have $E=EAA^{{\tiny{\textcircled{\tiny \dag}}}}$, which deduces that 
$$E_{2}=0.$$ 
Thus,   
\begin{equation*}
\begin{aligned}
 &A+E=U\left[\begin{matrix}
                  T+E_{1}& S\\
                 0&N
                  \end{matrix}
                  \right]U^*.      
\end{aligned}
\end{equation*}  
Hence,   
\begin{equation*}
\begin{aligned}
(A+E)^{{\tiny{\textcircled{\tiny \dag}}}}
&=U\left[\begin{matrix}
                  (T+E_{1})^{-1}& 0\\
                    0&0
                  \end{matrix}
                  \right]U^*\\
&=(I+A^{{\tiny{\textcircled{\tiny \dag}}}}E)^{-1}A^{{\tiny{\textcircled{\tiny \dag}}}}                  
\end{aligned}
\end{equation*}     
and 
\begin{equation*}
\begin{aligned}  
(A+E)^{{\tiny{\textcircled{\tiny \dag}}}} -A^{{\tiny{\textcircled{\tiny \dag}}}}=- A^{{\tiny{\textcircled{\tiny \dag}}}}E(A+E)^{{\tiny{\textcircled{\tiny \dag}}}}.
\end{aligned}
\end{equation*}  
Therefore, by Lemma \ref{1.2},  
$$\frac{\|(A+E)^{{\tiny{\textcircled{\tiny \dag}}}} -A^{{\tiny{\textcircled{\tiny \dag}}}}\|}{\|A^{{\tiny{\textcircled{\tiny \dag}}}}\|}\leq \frac{\|A^{{\tiny{\textcircled{\tiny \dag}}}}E\|}{1-\|A^{{\tiny{\textcircled{\tiny \dag}}}}E\|}.$$
It completes the proof.
\end{proof}

\begin{cor}Let $A,~E$ be as in Theorem \ref{theorem:3.4} and $\|A^{{\tiny{\textcircled{\tiny \dag}}}}\|\|E\|<1$. Then
  \begin{equation}\label{equation:3.3}
 \frac{\|(A+E)^{{\tiny{\textcircled{\tiny \dag}}}}-A^{{\tiny{\textcircled{\tiny \dag}}}}\|}{\|A^{{\tiny{\textcircled{\tiny \dag}}}}\|}\leq\frac{\kappa_{{\tiny{\textcircled{\tiny \dag}}}}(A)\|E\|\|A\|}{1-\kappa_{{\tiny{\textcircled{\tiny \dag}}}}(A)\|E\|\|A\|}.
 \end{equation}
 \end{cor}
The bound $(3.3)$ is perfectly analogous to the bounds for the Drazin inverse in \cite{W1997,C2000}, the Moore-Penrose inverse and the ordinary inverse in \cite{S1969}.

 

In the following, a sufficient condition for the continuity of the core-EP inverse is derived as a corollary.
\begin{cor}~Let $A\in \mathbb{C}^{n\times n}$ and let $\{E_j\}$ be a sequence of $n\times n$ matrices   such that $\|E_j\|\rightarrow 0$. If there exists a positive integer $j_0$ such that   
$E_j=E_jAA^{{\tiny{\textcircled{\tiny \dag}}}}=A^{{\tiny{\textcircled{\tiny \dag}}}}AE_j$ for $j\geq j_0$,
then $(A+E_j)^{{\tiny{\textcircled{\tiny \dag}}}}\rightarrow A^{{\tiny{\textcircled{\tiny \dag}}}}.$
\end{cor}

\begin{rem}~If $\mathrm{ind}(A)=1$, then the condition of Theorem \ref{theorem:3.4} is reduced to $E=EAA^{{\tiny{\textcircled{\tiny \#}}}}=A^{{\tiny{\textcircled{\tiny \#}}}}AE$
 and 
$\|A^{{\tiny{\textcircled{\tiny \#}}}}E\|<1.$ Thus, under these assumptions,  perturbation bounds for the core inverse are obtained.
\end{rem}

\subsection{The case: $AA^{{\tiny{\textcircled{\tiny \dag}}}}=(A+E)(A+E)^{{\tiny{\textcircled{\tiny \dag}}}}$ and~$A^{{\tiny{\textcircled{\tiny \dag}}}}A=(A+E)^{{\tiny{\textcircled{\tiny \dag}}}}(A+E)$}
In this part, perturbation bounds for the core-EP inverse are investigated
 under the assumption that $AA^{{\tiny{\textcircled{\tiny \dag}}}}=(A+E)(A+E)^{{\tiny{\textcircled{\tiny \dag}}}}$,~$A^{{\tiny{\textcircled{\tiny \dag}}}}A=(A+E)^{{\tiny{\textcircled{\tiny \dag}}}}(A+E)$. A sufficient condition for which the operation of the core-EP inverse is a continuous function  is derived as a corollary.
 \begin{thm}\label{3.1} Let $A,~E\in \mathbb{C}^{n\times n}$ such that
 $AA^{{\tiny{\textcircled{\tiny \dag}}}}=(A+E)(A+E)^{{\tiny{\textcircled{\tiny \dag}}}},~A^{{\tiny{\textcircled{\tiny \dag}}}}A=(A+E)^{{\tiny{\textcircled{\tiny \dag}}}}(A+E)$ and 
$
 \|A^{{\tiny{\textcircled{\tiny \dag}}}}E\|<1.
$
 Then 
 \begin{equation}
 \|(A+E)^{{\tiny{\textcircled{\tiny \dag}}}}\|\leq\frac{\|A^{{\tiny{\textcircled{\tiny \dag}}}}\|}{1-\|A^{{\tiny{\textcircled{\tiny \dag}}}}E\|}
 \end{equation}
and 
\begin{equation}
\frac{\|(A+E)^{{\tiny{\textcircled{\tiny \dag}}}}-A^{{\tiny{\textcircled{\tiny \dag}}}}\|}{\|A^{{\tiny{\textcircled{\tiny \dag}}}}\|}\leq \frac{\|A^{{\tiny{\textcircled{\tiny \dag}}}}E\|}{1-\|A^{{\tiny{\textcircled{\tiny \dag}}}}E\|}.
\end{equation}
\end{thm}

\begin{proof}~Since  $AA^{{\tiny{\textcircled{\tiny \dag}}}}=(A+E)(A+E)^{{\tiny{\textcircled{\tiny \dag}}}}~\text{and}~A^{{\tiny{\textcircled{\tiny \dag}}}}A=(A+E)^{{\tiny{\textcircled{\tiny \dag}}}}(A+E)$, then $$(A+E)^{{\tiny{\textcircled{\tiny \dag}}}}-A^{{\tiny{\textcircled{\tiny \dag}}}}=A^{{\tiny{\textcircled{\tiny \dag}}}}[A-(A+E)](A+E)^{{\tiny{\textcircled{\tiny \dag}}}}.$$ Thus, 
$(A+E)^{{\tiny{\textcircled{\tiny \dag}}}}=A^{{\tiny{\textcircled{\tiny \dag}}}}-A^{{\tiny{\textcircled{\tiny \dag}}}}E(A+E)^{{\tiny{\textcircled{\tiny \dag}}}}.$
Applying the  norm $\|\cdot\|$,   $$\|(A+E)^{{\tiny{\textcircled{\tiny \dag}}}}\|\leq\|A^{{\tiny{\textcircled{\tiny \dag}}}}\|+\|A^{{\tiny{\textcircled{\tiny \dag}}}}E\|\|(A+E)^{{\tiny{\textcircled{\tiny \dag}}}}\|.$$ Hence (3.4) is obtained since $
 \|A^{{\tiny{\textcircled{\tiny \dag}}}}E\|<1.
$

Again from $(A+E)^{{\tiny{\textcircled{\tiny \dag}}}}-A^{{\tiny{\textcircled{\tiny \dag}}}}=A^{{\tiny{\textcircled{\tiny \dag}}}}[A-(A+E)](A+E)^{{\tiny{\textcircled{\tiny \dag}}}}$, it follows that $$(A+E)^{{\tiny{\textcircled{\tiny \dag}}}}-A^{{\tiny{\textcircled{\tiny \dag}}}}=-A^{{\tiny{\textcircled{\tiny \dag}}}}E[A^{{\tiny{\textcircled{\tiny \dag}}}}+(A+E)^{{\tiny{\textcircled{\tiny \dag}}}}-A^{{\tiny{\textcircled{\tiny \dag}}}}].$$ Applying the norm  $\|\cdot\|$, 
$$\|(A+E)^{{\tiny{\textcircled{\tiny \dag}}}}-A^{{\tiny{\textcircled{\tiny \dag}}}}\|\leq\|A^{{\tiny{\textcircled{\tiny \dag}}}}E\|[\|A^{{\tiny{\textcircled{\tiny \dag}}}}\|+\|(A+E)^{{\tiny{\textcircled{\tiny \dag}}}}-A^{{\tiny{\textcircled{\tiny \dag}}}}\|].$$
Since  
$
 \|A^{{\tiny{\textcircled{\tiny \dag}}}}E\|<1
$, then (3.5) is derived. 

\end{proof}

\begin{cor}~Let $A,~E$ be as in Theorem \ref{3.1}
and $\|A^{{\tiny{\textcircled{\tiny \dag}}}}\|\|E\|< 1$. Then
\begin{equation}
\frac{\|(A+E)^{{\tiny{\textcircled{\tiny \dag}}}}-A^{{\tiny{\textcircled{\tiny \dag}}}}\|}{\|A^{{\tiny{\textcircled{\tiny \dag}}}}\|}\leq \frac{\kappa_{{\tiny{\textcircled{\tiny \dag}}}}(A)\|E\|/\|A\|}{1-\kappa_{{\tiny{\textcircled{\tiny \dag}}}}(A)\|E\|/\|A\|}.
\end{equation} 
\end{cor}

From Theorem \ref{3.1}, we derive a
sufficient condition  for the continuity of the core-EP inverse, as follows.
\begin{cor}~Let $A\in \mathbb{C}^{n\times n}$ and let $\{E_j\}$ be a sequence of $n\times n$ matrices   such that $\|E_j\|\rightarrow 0$. If there exists a positive integer $j_0$ such that   
$AA^{{\tiny{\textcircled{\tiny \dag}}}}=(A+E_j)(A+E_j)^{{\tiny{\textcircled{\tiny \dag}}}}$ and~$A^{{\tiny{\textcircled{\tiny \dag}}}}A=(A+E_j)^{{\tiny{\textcircled{\tiny \dag}}}}(A+E_j)$ for $j\geq j_0$,
then $(A+E_j)^{{\tiny{\textcircled{\tiny \dag}}}}\rightarrow A^{{\tiny{\textcircled{\tiny \dag}}}}.$
\end{cor}

\begin{rem}~If $\mathrm{ind}(A)=1$, then the condition of Theorem \ref{3.1} is reduced to $AA^{{\tiny{\textcircled{\tiny \#}}}}=(A+E)(A+E)^{{\tiny{\textcircled{\tiny \#}}}}$,~$A^{{\tiny{\textcircled{\tiny \#}}}}A=(A+E)^{{\tiny{\textcircled{\tiny \#}}}}(A+E)$
 and 
$\|A^{{\tiny{\textcircled{\tiny \#}}}}E\|<1.$ Thus, under these assumptions, a perturbation bound for the core inverse is obtained.
\end{rem}

\subsection{The case: $\mathrm{rank}\left(A^k\right)= \mathrm{rank}\left((A+E)^k\right)$}

It is known from \cite{R1982} that if $A$ and  $\{E_j\}$ are $n\times n$ matrices such that $\|E_j\|\rightarrow 0$, then there exists a positive integer $j_0$ such that $$\mathrm{rank}\left((A+E_j)^{k_j}\right)\geq \mathrm{rank}\left(A^{k_j}\right)$$ for $j\geq j_0$, where $k_j=$max$\{\mathrm{ind}(A),\mathrm{ind}(A+E_j)\}$.  

Let $A,~E\in \mathbb{C}^{n\times n}$. For an arbitrary positive integer $h$, define $E(A^h)$ by
$E(A^h)=(A+E)^h-A^h$. Then $\|(A+E)^h\|\leq \|A^h\|+\varepsilon(A^h),$ where 
$$\varepsilon(A^h)=\sum\limits^{h-1}\limits_{i=0}{\rm C}^{i}_{h}\|A\|^i\|E\|^{h-i}\geq\|E(A^h)\|$$ and ${\rm C}^{i}_{h}$ is the binomial coefficient.

\begin{lem}\label{1982}~Let $k=\mathrm{max}\{\mathrm{ind}(A),\mathrm{ind}(A+E)\}$. If
 $\mathrm{rank}\left((A+E)^k\right) > \mathrm{rank}\left(A^k\right)$, then
$$\|(A+E)^{{\tiny{\textcircled{\tiny \dag}}}}\| \geq \frac{1}{[\varepsilon(A^k)]^{1/k}}.$$
\end{lem}
\begin{proof}~The proof is analogous to the proof of \cite[Theorem 3]{R1982}. For completeness and convenience, we give a proof. 

Since $
\mathcal{R}(A^k)\oplus\mathcal{N}(A^k)=\mathbb{C}^{n}$, $\mathrm{rank}\left((A+E)^k\right) > \mathrm{rank}\left(A^k\right)$, then there exists $x\neq 0$ such that 
$x\in \mathcal{R}((A+E)^k)\cap\mathcal{N}(A^k)$ by \cite[Lemma 1]{C1975}. Without loss of generality, we can assume $\|x\|=1$. Then
\begin{equation*}
\begin{aligned}
1=x^*x=x^*[(A+E)^{{\tiny{\textcircled{\tiny \dag}}}}]^k(A+E)^kx=x^*[(A+E)^{{\tiny{\textcircled{\tiny \dag}}}}]^kE(A^k)x\leq \|(A+E)^{{\tiny{\textcircled{\tiny \dag}}}}\|^k\varepsilon(A^k).
\end{aligned}
\end{equation*} 
Hence $\|(A+E)^{{\tiny{\textcircled{\tiny \dag}}}}\|\geq \frac{1}{[\varepsilon(A^k)]^{1/k}}.$
\end{proof}
 Lemma \ref{1982} declares that $\|(A+E)^{{\tiny{\textcircled{\tiny \dag}}}}\|\rightarrow  \infty$ as
$\|E\| \rightarrow 0$ provided $\mathrm{rank}\left((A+E)^k\right) > \mathrm{rank}\left(A^k\right)$. Also, from Lemma~\ref{1982}  we immediately obtain the following result.

\begin{cor}\label{3.7}~Let $\{E_j\}$ be a sequence of $n\times n$ matrices   such that $\|E_j\|\rightarrow 0$, $k_j=\mathrm{max}\{\mathrm{ind}(A),\mathrm{ind}(A+E_j)\}$. If $(A+E_j)^{{\tiny{\textcircled{\tiny \dag}}}}\rightarrow A^{{\tiny{\textcircled{\tiny \dag}}}}$, then there exists  $j_0$ such that $\mathrm{rank}\left((A+E_j)^{k_j}\right)=\mathrm{rank}\left(A^{k_j}\right)$ for $j\geq j_0$.
\end{cor}

\begin{proof} Proof by contradiction.
\end{proof}

Thus, in this section,  to consider the perturbation bounds for the core-EP inverse, it sufficies to consider the case: 
$\mathrm{rank}\left(A^k\right)= \mathrm{rank}\left((A+E)^k\right)$.

\begin{lem}\emph{\cite{R1982}}\label{R}~Suppose $\mathrm{rank}\left((A+E)^h\right)=\mathrm{rank}\left(A^h\right)$ and $\|(A^h)^{\dag}\|\varepsilon(A^h)<1$. Then 
$$\|[(A+E)^h]^{\dag}\|\leq \frac{\|(A^h)^{\dag}\|}{1-\|(A^h)^{\dag}\|\varepsilon(A^h)}.$$
\end{lem}
Combine (1.2) with Lemma \ref{R}, then we have the following result.

\begin{thm}\label{3.9}~Suppose $\mathrm{ind}(A+E)=k$, $\mathrm{rank}\left((A+E)^k\right)=\mathrm{rank}\left(A^k\right)$ and $\|(A^{k+1})^{\dag}\|\varepsilon(A^{k+1})<1$. Then 
\begin{equation}
\|(A+E)^{{\tiny{\textcircled{\tiny \dag}}}}\|\leq \frac{(\|A^k\|+\varepsilon(A^{k}))\|(A^{k+1})^{\dag}\|}{1-\|(A^{k+1})^{\dag}\|\varepsilon(A^{k+1})}.
\end{equation}
\end{thm}

Theorem \ref{3.9} states that $(A+E)^{{\tiny{\textcircled{\tiny \dag}}}}$ is bounded
provided~$$\mathrm{rank}\left((A+E)^{\mathrm{ind}(A+E)}\right)=\mathrm{rank}\left(A^{\mathrm{ind}(A+E)}\right).$$ This is one of the bases for obtaining the perturbation bound for the core-EP inverse. The other one is contained in the asymptotic expansion of $(A+E)^{{\tiny{\textcircled{\tiny \dag}}}}-A^{{\tiny{\textcircled{\tiny \dag}}}}$.  

Let $k=$max$\{\mathrm{ind}(A),\mathrm{ind}(A+E)\}$. Then
\begin{equation}
\begin{aligned}
(A+E)^{{\tiny{\textcircled{\tiny \dag}}}}-A^{{\tiny{\textcircled{\tiny \dag}}}}
&=-(A+E)^{{\tiny{\textcircled{\tiny \dag}}}}EA^{{\tiny{\textcircled{\tiny \dag}}}}
+(A+E)^{{\tiny{\textcircled{\tiny \dag}}}}-A^{{\tiny{\textcircled{\tiny \dag}}}}+(A+E)^{{\tiny{\textcircled{\tiny \dag}}}}(A+E-A)A^{{\tiny{\textcircled{\tiny \dag}}}}\\
&=-(A+E)^{{\tiny{\textcircled{\tiny \dag}}}}EA^{{\tiny{\textcircled{\tiny \dag}}}}
+(A+E)^{{\tiny{\textcircled{\tiny \dag}}}}(I-AA^{{\tiny{\textcircled{\tiny \dag}}}})
-[I-(A+E)^{{\tiny{\textcircled{\tiny \dag}}}}(A+E)]A^{{\tiny{\textcircled{\tiny \dag}}}}\\
&=-(A+E)^{{\tiny{\textcircled{\tiny \dag}}}}EA^{{\tiny{\textcircled{\tiny \dag}}}}
+(A+E)^{{\tiny{\textcircled{\tiny \dag}}}}[(A+E)^{{\tiny{\textcircled{\tiny \dag}}}}]^{k*}[E(A^k)]^*(I-AA^{{\tiny{\textcircled{\tiny \dag}}}})\\
&~~~~+[I-(A+E)^{{\tiny{\textcircled{\tiny \dag}}}}(A+E)]E(A^k)(A^{{\tiny{\textcircled{\tiny \dag}}}})^{k+1}.
\end{aligned}
\end{equation}
Take $\|\cdot\|$ on (3.8), then
\begin{equation}
\begin{aligned}
\|(A+E)^{{\tiny{\textcircled{\tiny \dag}}}}-A^{{\tiny{\textcircled{\tiny \dag}}}}\|
&\leq \|(A+E)^{{\tiny{\textcircled{\tiny \dag}}}}\|\|A^{{\tiny{\textcircled{\tiny \dag}}}}\|\|E\|+\\
&~~~~\|(A+E)^{{\tiny{\textcircled{\tiny \dag}}}}\|^{k+1}(I+\|A\|\|A^{{\tiny{\textcircled{\tiny \dag}}}}\|)\varepsilon(A^k)+\\
&~~~~[I+\|(A+E)^{{\tiny{\textcircled{\tiny \dag}}}}\|\|A\|+\|(A+E)^{{\tiny{\textcircled{\tiny \dag}}}}\|\|E\|]\|A^{{\tiny{\textcircled{\tiny \dag}}}}\|^{k+1}\varepsilon(A^k).
\end{aligned}
\end{equation}
Now suppose $\|(A^{k+1})^{\dag}\|\varepsilon(A^{k+1})<1$ and $\mathrm{rank}\left((A+E)^k\right)=\mathrm{rank}\left(A^k\right)$, then by Theorem \ref{3.9},  $(A+E)^{{\tiny{\textcircled{\tiny \dag}}}}$ is bounded. 
Thus, from Equality (3.9), 
$\|(A+E)^{{\tiny{\textcircled{\tiny \dag}}}}-A^{{\tiny{\textcircled{\tiny \dag}}}}\|\rightarrow 0$ as $\|E\|\rightarrow 0$, that is to say,
\begin{equation}
(A+E)^{{\tiny{\textcircled{\tiny \dag}}}}=A^{{\tiny{\textcircled{\tiny \dag}}}}+O(\|E\|).
\end{equation}
In order to derive the perturbation bound,  we plug (3.10) into the right side of (3.8). Then
\begin{equation}
\begin{aligned}
(A+E)^{{\tiny{\textcircled{\tiny \dag}}}}-A^{{\tiny{\textcircled{\tiny \dag}}}}
&=-A^{{\tiny{\textcircled{\tiny \dag}}}}EA^{{\tiny{\textcircled{\tiny \dag}}}}+
A^{{\tiny{\textcircled{\tiny \dag}}}}[(A^{{\tiny{\textcircled{\tiny \dag}}}})^k]^*(\sum\limits^{k-1}\limits_{i=0}A^iEA^{k-1-i})^*(I-AA^{{\tiny{\textcircled{\tiny \dag}}}})\\
&~~~~+(I-A^{{\tiny{\textcircled{\tiny \dag}}}}A)\sum\limits^{k-1}\limits_{i=0}A^iEA^{k-1-i}(A^{{\tiny{\textcircled{\tiny \dag}}}})^{k+1}+O(\|E\|^2)\\
&= -A^{{\tiny{\textcircled{\tiny \dag}}}}EA^{{\tiny{\textcircled{\tiny \dag}}}}+
A^{{\tiny{\textcircled{\tiny \dag}}}}[\sum\limits^{k-1}\limits_{i=0}A^iE(A^{{\tiny{\textcircled{\tiny \dag}}}})^{i+1}]^*(I-AA^{{\tiny{\textcircled{\tiny \dag}}}})\\
&~~~~+(I-A^{{\tiny{\textcircled{\tiny \dag}}}}A)\sum\limits^{k-1}\limits_{i=0}A^iE(A^{{\tiny{\textcircled{\tiny \dag}}}})^{i+2}+O(\|E\|^2).
\end{aligned}
\end{equation}
Take $\|\cdot\|$ on (3.11), then we obtain the following result.

\begin{thm}\label{3.10}~Let $k=\mathrm{max}\{\mathrm{ind}(A),\mathrm{ind}(A+E)\}$, $\mathrm{rank}\left((A+E)^k\right)=\mathrm{rank}\left(A^k\right)$ and \\
$\|(A^{k+1})^{\dag}\|\varepsilon(A^{k+1})<1$. 
Then \begin{equation}
\begin{aligned}
\frac{\|(A+E)^{{\tiny{\textcircled{\tiny \dag}}}}-A^{{\tiny{\textcircled{\tiny \dag}}}}\|}{\|A^{{\tiny{\textcircled{\tiny \dag}}}}\|}\leq C(A)\frac{\|E\|}{\|A\|}+o(\|E\|^2),
\end{aligned}
\end{equation}
where $C(A)=[2\sum\limits^{k-1}\limits_{i=0}\|A\|^i\|A^{{\tiny{\textcircled{\tiny \dag}}}}\|^{i+1}(1+\|A\|\|A^{{\tiny{\textcircled{\tiny \dag}}}}\|)+\|A^{{\tiny{\textcircled{\tiny \dag}}}}\|]\|A\|$.
\end{thm}

\begin{cor}\label{3.11}~Let $\{E_j\}$ be a sequence of $n\times n$ matrices   such that $\|E_j\|\rightarrow 0$ and let  $k_j=\mathrm{max}\{\mathrm{ind}(A),\mathrm{ind}(A+E_j)\}$. If there exists  $j_0$ such that  $\mathrm{rank} \left((A+E_j)^{k_j}\right)=\mathrm{rank} \left(A^{k_j}\right)$ for  $j\geq j_0$, then $(A+E_j)^{{\tiny{\textcircled{\tiny \dag}}}}\rightarrow A^{{\tiny{\textcircled{\tiny \dag}}}}$.
\end{cor}

``Let $k=\mathrm{max}\{\mathrm{ind}(A),\mathrm{ind}(A+E)\}$,  $\mathrm{rank}\left((A+E)^{k}\right)=\mathrm{rank}\left(A^{k}\right)$" is the same meaning as ``$\mathrm{rank} \left((A+E)^{\mathrm{ind}(A+E)}\right)=\mathrm{rank} \left(A^{\mathrm{ind}(A)}\right)$". Thus, 
Corollary \ref{3.11} in conjunction with Corollary \ref{3.7} gives another proof for  the equivalence of $(1)$ and $(3)$ in Theorem \ref{2.6}, which means that $\mathrm{rank} \left((A+E)^k\right)=\mathrm{rank} \left(A^k\right)$ is the weakest condition for the continuity of the core-EP inverse. 

\begin{rem}~If $k=1$, then the condition of Theorem \ref{3.10} becomes $\mathrm{rank} (A+E)=\mathrm{rank}(A)$ and 
$\|(A^{2})^{\dag}\|\varepsilon(A^{2})<1$. Thus, under these assumptions, we derive a perturbation bound for the core inverse.
\end{rem}

\section{Numerical examples}
In this section, we shall establish a numerical example to compare the upper bounds for  $\frac{\|(A+E)^{{\tiny{\textcircled{\tiny \dag}}}}-A^{{\tiny{\textcircled{\tiny \dag}}}}\|}{\|A^{{\tiny{\textcircled{\tiny \dag}}}}\|}$ derived in   
(3.5) and (3.12). Let $$A=\begin{bmatrix}
                1&0 &0\\
                0&0 &0\\
                0&0 &0
               \end{bmatrix},~E=\begin{bmatrix}
                                               \varepsilon&\varepsilon&0\\
                                               0              &0                &0 \\
                                               0               &0                &0
                                     \end{bmatrix}.$$
Then $\mathrm{ind}(A)=\mathrm{ind}(A+E)=1$, rank$(A)=$rank$(A+E)=1$ and $AA^{{\tiny{\textcircled{\tiny \dag}}}}=(A+E)(A+E)^{{\tiny{\textcircled{\tiny \dag}}}}$, $A^{{\tiny{\textcircled{\tiny \dag}}}}A=(A+E)^{{\tiny{\textcircled{\tiny \dag}}}}(A+E)$.  Thus $A$ and $E$ satisfy the conditions in Theorems \ref{3.1} and \ref{3.10}.                                 
Table 1 shows that our bound (3.5) is slightly better than (3.12).                                     
\begin{table}[!htp]
\caption{Comparison  of upper bounds of  $\|(A+E)^{{\tiny{\textcircled{\tiny \dag}}}}-A^{{\tiny{\textcircled{\tiny \dag}}}}\|/\|A^{{\tiny{\textcircled{\tiny \dag}}}}\|$}
\begin{tabular}{ccccc}
\hline
~~~~~&$\varepsilon=0.1000$~&~$\varepsilon=0.0100$~~&~~$\varepsilon=0.0010$~~&~~$\varepsilon=0.0001$~~~~\\
\hline
Exact&0.0909~~~~~~&0.0099~~~~~~&0.0010~~~~~&1.0000e-04\\
(3.5)&0.1647~~~~~~&0.0143~~~~~~&0.0014~~~~~&1.4143e-04\\
(3.12)&~0.7070+$o(\|E\|^2)$&~0.0705+$o(\|E\|^2)$&~~0.0070+$o(\|E\|^2)$&~~~7.0710e-04+$o(\|E\|^2)$\\
\hline
\end{tabular}
\end{table}

\section{Applications to semistable matrices}
Following \cite{C2001}, a matrix $A\in \mathbb{C}^{n\times n}$ is called semistable if ind$(A)\leq 1$ and the nonzero eigenvalues $\lambda$ of $A$ satisfy Re $\lambda<0$; a semistable matrix with ind$(A)=0$ is stable. It is known that we have an integral representation for the inverse of $A$ if $A$ is stable (for example, see \cite{C2001}):
\begin{equation}
A^{-1}=-\int^{\infty}_0 \mathrm{exp}(tA)\mathrm{d}t.
\end{equation}
In this section, an integral representation for the core-EP inverse of a perturbed matrix $A+E$ is discussed under the condition  $E=EAA^{{\tiny{\textcircled{\tiny \dag}}}}=A^{{\tiny{\textcircled{\tiny \dag}}}}AE$, where $A$ is a semistable matrix. 

\begin{lem}\emph{\cite{C2001}}\label{4.1}~Let $A\in \mathbb{C}^{n\times n}$ be stable. If there exists $\eta>0$ such that $\|E\|<\eta$, then $A+E$ is stable.
\end{lem}

\begin{thm}~Let $A\in \mathbb{C}^{n\times n}$ be semistable and let $E\in \mathbb{C}^{n\times n}$ such that $E=EAA^{{\tiny{\textcircled{\tiny \dag}}}}=A^{{\tiny{\textcircled{\tiny \dag}}}}AE$. Then there exists $\delta (A)>0$ such that for $\|E\|<\delta (A)$, 
\begin{equation}
(A+E)^{{\tiny{\textcircled{\tiny \dag}}}}=-\int^{\infty}_0\mathrm{exp}(t(A+E))AA^{{\tiny{\textcircled{\tiny \dag}}}}\mathrm{d}t.
\end{equation}
\end{thm}

\begin{proof}~For $A\in \mathbb{C}^{n\times n}$, then we have $$A=U\begin{bmatrix}
           T&S\\
           0&N
          \end{bmatrix}U^*$$ as in (2.5), where $U$ is unitary, $T$ is nonsingular and $N$ is nilpotent. 
          
  From the assumption $E=EAA^{{\tiny{\textcircled{\tiny \dag}}}}=A^{{\tiny{\textcircled{\tiny \dag}}}}AE$, it follows that $E=U\begin{bmatrix}
           E_1&0\\
           0&0
          \end{bmatrix}U^*$ according to the proof of Theorem \ref{theorem:3.4}. 
Then $\|E_1\|=\|U^*EU\|\leq \|U^{-1}\|\|E\|\|U\|=\kappa(U)\|E\|$.  Observe that 
\begin{equation*}
\begin{aligned}
\mathrm{exp}(t(A+E))AA^{{\tiny{\textcircled{\tiny \dag}}}}&=U\begin{bmatrix}
          \mathrm{exp}( t(T+E_1))&\Delta\\
           0&\mathrm{exp}(tN)
          \end{bmatrix}
          \begin{bmatrix}
           I&0\\
           0&0
          \end{bmatrix}U^*\\
          &=U\begin{bmatrix}
          \mathrm{exp}( t(T+E_1))&0\\
           0&0
          \end{bmatrix}U^*.
\end{aligned}
\end{equation*}     
          
Since $A$ is semistable,  $T$ is stable. Set $\delta(A)=\frac{\eta}{\kappa(U)}$, if $\|E\|<\delta(A)$, then $\|E_1\|<\eta$, thus $T+E_1$ is stable by Lemma \ref{4.1}.  Therefore $T+E_1$ is integrable  on the interval $[0,\infty)$.     
In view of (5.1),
\begin{equation*}
\begin{aligned}
-\int^{\infty}_0\mathrm{exp}(t(A+E))AA^{{\tiny{\textcircled{\tiny \dag}}}}\mathrm{d}t
&=-U\begin{bmatrix}
          \int^{\infty}_0\mathrm{exp}( t(T+E_1))\mathrm{d}t&0\\
           0&0
          \end{bmatrix}U^*\\
          &=U\begin{bmatrix}
          (T+E_1)^{-1}&0\\
           0&0
          \end{bmatrix}U^*=(A+E)^{{\tiny{\textcircled{\tiny \dag}}}}.          
\end{aligned}
\end{equation*}
It completes the proof.          
\end{proof}

\vspace{8mm}
\noindent {\large\bf Acknowledgements}\\
This research is supported by the National Natural Science Foundation
of China (No.11771076), the Scientific Innovation Research of College Graduates in Jiangsu Province (No.KYZZ16$\_$0112), partially supported  by   FCT- `Funda\c{c}\~{a}o para a Ci\^{e}ncia e a Tecnologia', within the project UID-MAT-00013/2013.

\end{document}